\documentclass[12pt]{article}

\usepackage{amsmath} 
\usepackage{amssymb}
\usepackage{amstext}
\usepackage{amsthm}
\usepackage{amsfonts}
\usepackage{mathtools}
\usepackage{enumerate}  
\usepackage{dsfont}
\usepackage[utf8]{inputenc}

\usepackage[pdftex]{graphicx} 
\usepackage{hyperref}

\theoremstyle{plain}

\newtheorem{theorem}{Theorem}
\newtheorem*{theorem*}{Theorem}
\newtheorem{lemma}[theorem]{Lemma}
\newtheorem{corollary}[theorem]{Corollary}

\usepackage{notation}

\title{Quantum walks and the size of the graph}
\author{Gabriel Coutinho \\ \small{Dept. Computer Science}\\ \small{Universidade Federal de Minas Gerais} \\ \small{Belo Horizonte, Brazil} \\ \texttt{gabriel@dcc.ufmg.br}}

\date{\today}

\begin{document}

\maketitle

\begin{abstract}
A continuous-time quantum walk is modelled using a graph. In this short paper, we provide lower bounds on the size of a graph that would allow for some quantum phenomena to occur. Among other things, we show that, in the adjacency matrix quantum walk model, the number of edges is bounded below by a cubic function on the eccentricity of a periodic vertex. This gives some idea on the shape of a graph that would admit periodicity or perfect state transfer. We also raise some extremal type of questions in the end that could lead to future research.
\end{abstract}
\ \\
\texttt{Keywords: quantum walk; spectral bounds; state transfer.} \\
\texttt{MSC: 05C50; 81P68.}

\section{Introduction}
We model a network of interacting qubits in the continuous-time quantum walk XY model by a simple graph $G$ with adjacency matrix $A$, which encodes the pairs of qubits that interact. We are interested in the single-excitation subspace, that is the case when the system is initialized with one qubit in state $\ket 1$ and all others in state $\ket 0$. This model with a time-independent Hamiltonian evolves according to Schrödinger equation, and the main problem we address is what happens to the state $\ket 1$ as time passes. For example, if this state is observed at another vertex with probability $1$ after a certain time, we say that \textit{perfect state transfer} has happened. This is equivalent to having a pair of vertices $a$ and $b$ and a time $\tau \geq 0$ such that
\[|\exp(\ii \tau A)_{a,b}| = 1.\]
When $a = b$ in the equation above, we say that vertex $a$ is periodic. Vertices involved in perfect state transfer are periodic at double the time, but the converse does not hold in general.

Perfect state transfer was first considered by Bose \cite{BoseQuantumComPaths} and has since been studied in a good number of papers and in different contexts, spanning the fields of physics, mathematics and computer science. Surveys are found in Kendon and Tamon \cite{KendonTamon} and Godsil \cite{GodsilStateTransfer12}. The thesis \cite{CoutinhoPhD} contains a detailed introduction and a more recent compilation of known results.

Christandl et al.~\cite{ChristandlPSTQuantumSpinNet2} showed that the paths on two and three vertices admit perfect state transfer between the end vertices, and that if a graph admits perfect state transfer, then any Cartesian power of this graph also admits perfect state transfer at the same time. The pictures below depict the Cartesian squares and cubes of paths on 2 and 3 vertices, and the vertices in black are involved in perfect state transfer.

\begin{center}
\includegraphics[scale=0.4]{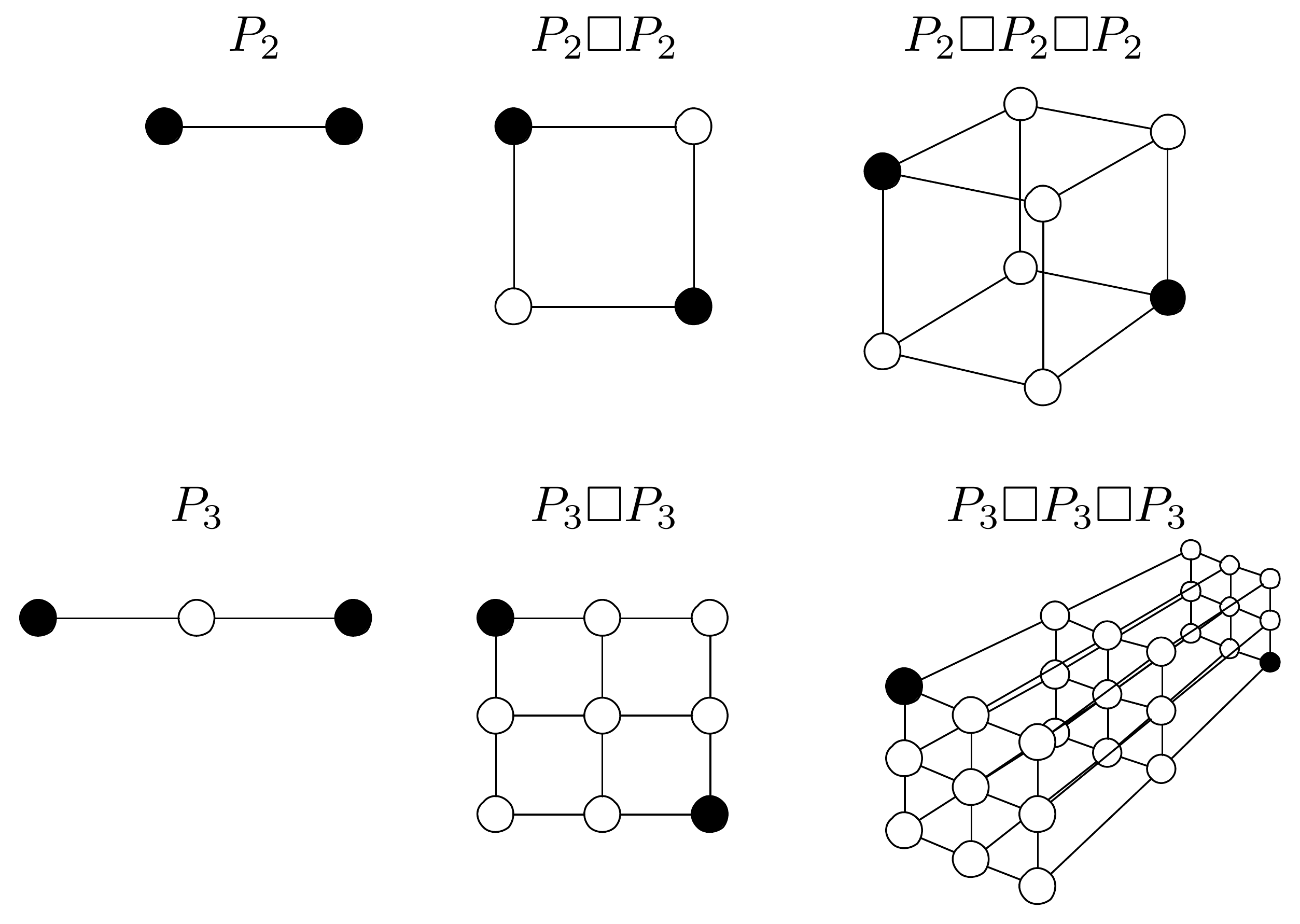}
\end{center}

Suppose a graph $G$ admits perfect state transfer between vertices at distance $d$. In this paper, we are interested in determining what is the cost in terms of the number of vertices or edges. To this day, the best known trade-off is achieved by the Cartesian powers of $P_3$, in which perfect state transfer between vertices at even distance $d$ happens in a graph with $3^{d/2}$ vertices and $d \cdot 3^{(d/2) - 1}$ edges. However, there is no known lower bound on the number of edges in terms of the diameter alone other than a trivial linear bound. See \cite[section E]{KayPerfectcommunquantumnetworks} for some discussion.

As we mentioned, vertices involved in perfect state transfer must be periodic. This turns out to be the key concept needed to bound the size of the graph. More specifically, we will derive bounds that depend on the eccentricity of a periodic vertex, and also on the period.

\section{The eccentricity of periodic vertices}

In our considerations below, $M$ is a symmetric integer matrix whose rows and columns are indexed by the vertices of a graph, and whose an off diagonal entry is non-zero if and only if the corresponding pair of vertices is adjacent. Moreover, we assume the sign of all off-diagonal entries is the same. For example, $M$ can be the adjacency, the Laplacian, or the signless Laplacian, as well as weighted versions of these matrices, provided all weights are integers with the same sign.

We say that vertex $a$ is periodic in $G$ according to $M$ if there is a positive real time $\tau$ such that $\exp(\ii \tau M)_{a,a}$ has absolute value equal to one.

Let $\theta_0 > ... > \theta_t$ be the distinct eigenvalues of $M$. If $a$ is a vertex of $G$, we denote by $\ee_a$ the $01$-vector in $\R^n$ that is $0$ everywhere except for the position corresponding to $a$. Because $M$ is symmetric, it has a spectral decomposition into orthogonal idempotents as follows.
\[M = \sum_{r = 0}^t \theta_r E_r.\]
Let $\Phi_a$ denote the set of eigenvalues of $M$ such that $\theta_r \in \Phi_a$ if and only if $E_r \ee_a \neq 0$. This is the eigenvalue support of $a$.

\begin{theorem}[Godsil \cite{GodsilPerfectStateTransfer12}, Theorem 6.1]
If vertex $a$ is periodic in $G$ according to $M$, then the non-zero elements in $\Phi_a$ are either all integers or all quadratic integers. Moreover, there is a square-free positive integer $\Delta$, an integer $\alpha$ and integers $\beta_r$ such that
\[\theta_r \in \Phi_a \implies \theta_r = \frac{1}{2} (\alpha+\beta_r\sqrt{\Delta}).\]
We consider $\Delta = 1$ for the cases where all eigenvalues are integers.
\end{theorem}

If $M = L(G)$, note that $0$ is always an eigenvalue in the eigenvalue support of any vertex. Moreover, $L(G) \succeq 0$, and so no eigevalue of $L(G)$ can be of the form $b \sqrt{\Delta}$ with $b>0$ and $\Delta > 1$, as this would imply that $-b\sqrt{\Delta}$ is also an eigenvalue. As a consequence, we have the following corollary.

\begin{corollary}
If $M = L(G)$ and vertex $a$ is periodic, then the elements in $\Phi_a$ are all integers.
\end{corollary}

We will also need the following result.

\begin{theorem}[Coutinho \cite{CoutinhoPhD}, Theorem 2.4.4] \label{thm:1}
Suppose vertex $a$ is periodic in $G$ according to $M$ at time $\tau$. Let 
\[ g = \gcd\left( \left\{ \frac{\theta_0 - \theta_r}{\sqrt{\Delta}} \right\}_{\theta_r \in \Phi_a} \right).\]
Then $\tau$ must be an odd multiple of $\dfrac{2 \pi}{g \sqrt{\Delta}}$.
\end{theorem}

As a consequence of the Theorem above, the minimum time periodicity occurs is at most $2\pi$.

As an immediate consequence, we have the following lemma.

\begin{lemma} \label{lem:1}
Suppose vertex $a$ is periodic in $G$ according to $M$ at minimum time $\tau$, and let $\theta$ and $\theta'$ be distinct eigenvalues in $\Phi_a$. Then
\[\theta - \theta' \ \geq\  \frac{2\pi}{\tau}.\]
\end{lemma}
\begin{proof}
Suppose $\theta$ and $\theta'$ are two elements of $\Phi_a$. Then, by Theorem \ref{thm:1}, there is an integer $k$ such that
\[\frac{2\pi}{\tau} \ k =  \theta - \theta'.\]
\end{proof}

Let $\varepsilon_a$ denote the eccentricity of vertex $a$ in the graph $G$, that is, the maximum distance between any vertex of $G$ and vertex $a$. A standard argument in algebraic graph theory leads to a relation between $\varepsilon_a$ and $|\Phi_a|$.
\begin{lemma} \label{lem:2}
	Let $M$ be as defined in the beginning of this section. Let $\Phi_a$ be the  eigenvalue support of vertex $a$ according to $M$. Then $\varepsilon_a \leq |\Phi_a|$.
\end{lemma}
\begin{proof}
	Consider the subspace of $\R^n$ defined as
	\[W_a = \langle \{ M^i \ee_a\}_{i \geq 0} \rangle.\]
	It follows that
	\[W_a = \langle \{ E_r \ee_a \}_{\theta_r \in \Phi_a} \rangle,\]
	hence $\dim W_a = |\Phi_a|$. Because the non-zero off-diagonal entries of $M$ correspond to adjacent vertices and have all the same sign, it follows that the vectors $\{M^i \ee_a\}_{i = 0}^{\varepsilon_a}$ are all independent, and thus
	\[\varepsilon_a +1\leq |\Phi_a|.\]
\end{proof}

We will now proceed to establish bounds on the size of the graph that depend solely on the eccentricity of a periodic vertex.

\begin{theorem}
	Let $G$ be a graph with $m$ edges that, according to the quantum walk model defined by the adjacency matrix, contains a periodic vertex $a$ with period $\tau$. Let $\varepsilon_a$ be the eccentricity of $a$. Then
	\[\left(\frac{\varepsilon_a}{3}\right)^3 < 2m.\]
\end{theorem}
\begin{proof}
	Let $\theta_0,...,\theta_{n-1}$ denote the eigenvalues of $A$, possibly with repetition, and assume they are ordered in such a way that $\theta_0^2 \geq \theta_1^2 \geq ... \geq \theta_{n-1}^2$. Because the diagonal entries of $A^2$ contain the degrees of each vertex, it follows that
	\[\tr A^2 = 2m = \sum_{j = 0}^{n-1} \theta_j^2.\]
	As a consequence
	\[\theta_j^2 \leq \frac{2m}{j+1}.\]
	Let $k = \lfloor \sqrt[3]{2m} \rfloor$. Note that $k \leq n-1$. From Lemma \ref{lem:1}, the separation between any two distinct eigenvalues in $\Phi_a$ is at least $2\pi / \tau$. Thus, in the worst case where all eigenvalues $\theta_0,...,\theta_{k-1}$ belong to $\Phi_a$, we can still say that
	\[\frac{2 \pi}{\tau} (|\Phi_a| - k - 1) \leq 2|\theta_k|.\]
	Hence
	\begin{align}
	|\Phi_a| & \leq \frac{\tau}{\pi} \sqrt{\frac{2m}{k+1}} + k + 1 \\ & 
	< \sqrt[3]{2m}\left( \frac{\tau}{\pi} + 1\right) + 1 \\ & 
	\leq 3 \sqrt[3]{2m} + 1.
	\end{align}
	where the last inequality follows from $\tau \leq 2 \pi$ (Theorem \ref{thm:1}). We know from Lemma \ref{lem:2} that $\varepsilon_a + 1 \leq |\Phi_a|$, thus
	\[\left(\frac{\varepsilon_a}{3}\right)^3 < 2m.\]
\end{proof}

It is immediate from the theorem above that the number of vertices cannot be bound above by a linear function on $\varepsilon_a$. Adapting the proof to matrices $M$ as defined in the beginning of this section can lead to similar bounds in other quantum walk models. However, the specificities of the matrix could make the bound stronger or weaker. For instance, with the Laplacian matrix, the same technique allows only for a quadratic bound.

In particular, we know that all eigenvalues $0 = \lambda_{n-1} \leq ... \leq \lambda_0$ are non-negative and that
\[\sum_{j = 0}^{n-1} \lambda_j = 2m.\]
So we have the following.

\begin{theorem}
	Let $G$ be a graph with $m$ edges that, according to the quantum walk model defined by the Laplacian matrix, contains a periodic vertex $a$ with period $\tau$. Let $\varepsilon_a$ be the eccentricity of $a$. Then
	\[\left(\frac{\varepsilon_a}{3}\right)^2 < m.\]
\end{theorem}
\begin{proof}
	We have
	\[\lambda_j \leq \frac{2m}{j+1}.\]
	Let $k = \lfloor \sqrt{m} \rfloor$. From Lemma \ref{lem:1}, the separation between an two distinct eigenvalues in $\Phi_a$ is at least $2\pi / \tau$. Thus, in the worst case where all eigenvalues $\lambda_0,...,\lambda_{k-1}$ belong to $\Phi_a$, we can still say that
	\[\frac{2 \pi}{\tau} (|\Phi_a| - k - 1) \leq \lambda_k.\]
	Hence
	\begin{align}
	|\Phi_a| & \leq \frac{\tau}{\pi} \cdot \frac{m}{k+1} + k + 1 \\ & 
	< \sqrt{m}\left( \frac{\tau}{\pi} + 1\right) + 1 \\ & 
	\leq 3 \sqrt{m} + 1.
	\end{align}
	where the last inequality follows from $\tau \leq 2 \pi$ (Theorem \ref{thm:1}). We know from Lemma \ref{lem:2} that $\varepsilon_a + 1\leq |\Phi_a|$, thus
	\[\left(\frac{\varepsilon_a}{3}\right)^2 < m.\]
\end{proof}

\section{Discussion and questions}

The task of finding vertices admitting perfect state transfer at arbitrarily large distances has been studied now for at least 10 years. As far as I know, this paper provides the first non-trivial lower bounds on the size of an unweighed graph that would admit perfect state transfer relative to its diameter. This is because any vertex involved in perfect state transfer must be periodic, and the eccentricity of any vertex is at least half the diameter of the graph. One of the reasons that make this work relevant is that the size of the graph can be seen as the cost to construct the quantum system.

Our cubic (or quadratic, depending on the model) lower bound frustrates the expectation that small variations on arbitrarily long paths could provide families of graphs admitting perfect state transfer. It also shows that some trees do not admit perfect state transfer, providing some support to the conjecture that, in the adjacency matrix model, no tree other than $P_2$ or $P_3$ does (see \cite{CoutinhoLiu2}).

The dual line of investigation to what we expose here is to find families of relatively small graphs admitting perfect state transfer at arbitrarily large distances. As we pointed in the introduction, the best trade-off is achieved by Cartesian powers of $P_3$. In \cite[Section 5]{CoutinhoSpectrallyExtremal2}, some strategy on how to perturb these powers to reduce the size of the graph are briefly discussed, but we were not able to obtain a relevant asymptotic improvement.

We end with a list of questions related to our work.

\begin{enumerate}
	\item If there is a periodic vertex at time $2\pi/\lambda$, provide a lower bound on the size of the graph that depends on $\lambda$ better than the cubic bound we have. I am guessing this could be an exponential in $\lambda$. Note that the only examples we know of perfect state transfer happening at arbitrarily small times are coming from graphs which are very large (see \cite{AdaChanComplexHadamardIUMPST}).
	\item Improve our polynomial bounds to exponential, or find a family of graphs admitting perfect state transfer at arbitrarily large distances and size bounded above by a polynomial in their diameter.
\end{enumerate}

\section{Acknowledgements}

I thank Chris Godsil and Alastair Kay for some discussion on the topic of this paper. I also acknowledge a travel grant from the Dept. of Computer Science at UFMG.

\end{document}